\documentclass[12pt]{amsart}
\pdfoutput=1
\usepackage[headings]{fullpage}
\usepackage{amssymb, epic,eepic,epsfig}
\usepackage{graphicx}
\usepackage{texdraw}
\usepackage{url}
\usepackage[bookmarks=true,%
    colorlinks=true,%
    linkcolor=blue,%
    citecolor=blue,%
    filecolor=blue,%
    menucolor=blue,%
    urlcolor=blue,%
    breaklinks=true]{hyperref}
\usepackage{tikz,cancel,color,hyphenat,multicol} 
\usepackage[all]{xy}
\usepackage[english]{babel}

\newtheorem{theorem}{Theorem}[section]
\theoremstyle{definition}
\newtheorem{proposition}[theorem]{Proposition}
\newtheorem{lemma}[theorem]{Lemma}
\newtheorem{definition}[theorem]{Definition}
\newtheorem{remark}[theorem]{Remark}
\newtheorem{corollary}[theorem]{Corollary}

\newtheorem{example}[theorem]{Example}

\def\calB{\mathcal{B}}

\def\calS{\mathcal S}

\def\calF{{\mathcal F}}

\def\calV{{\mathcal V}}

\def\calM{{\mathcal M}}
\def\bbR{{\mathbb R}}

\def\cal{\mathcal}

\def\-{\setminus}
\def\red#1{\textcolor[rgb]{1.00,0.00,0.00}{#1}}%

\begin{document}


\title[Construction of the graphic matroid from the lattice of integer flows]{
A construction of the graphic matroid from the lattice of integer flows}

\author{Zsuzsanna Dancso}
\address{Australian National University \\
John Dedman Building \\
Union Lane, Acton, ACT 2602 Australia \newline
         {\tt \url{http://www.math.toronto.edu/~zsuzsi }}}
\email{zsuzsanna.dancso@anu.edu.au}
\author{Stavros Garoufalidis}
\address{School of Mathematics \\
         Georgia Institute of Technology \\
         Atlanta, GA 30332-0160, USA \newline 
         {\tt \url{http://www.math.gatech.edu/~stavros }}}
\email{stavros@math.gatech.edu}
\thanks{
{\em Key words and phrases: graphs, lattice of flows, lattices, Voronoi cells.
}
}

\date{November 9, 2016}


\begin{abstract}
The lattice of integer flows of a graph is known to determine the graph up 
to 2-isomorphism (work of Su--Wagner and Caporaso--Viviani). In this paper 
we give an algorithmic construction of the graphic matroid $\calM(G)$ of a 
graph $G$, given its lattice of integer flows $\calF(G)$. The algorithm can 
then be applied to compute any other 2-isomorphism invariants (that is, 
matroid invariants) of $G$ from $\calF(G)$. Our method is based on a result 
of Amini which describes the relationship between the geometry of
the Voronoi cell of $\calF(G)$ and the structure of $G$. 
\end{abstract}

\maketitle

{\footnotesize
\tableofcontents
}

\section{Introduction}
\label{sec.intro}

The goal of this paper is to reconstruct, to the extent possible, a 
two-edge-connected graph $G$ from its lattice of integer flows $\calF(G)$. 
The lattice $\calF(G)$ does not contain all the information about the graph: 
it is invariant under graph {\em 2-isomorphisms}, but it separates 
2-isomorphism classes of graphs 
(a result by Caporaso--Viviani 
\cite[Thm 3.1.1]{Caporaso}, Su--Wagner \cite{SuWagner}, implicitly Mumford 
in \cite{TadaoSeshadri}, and Artamkin \cite{Art} for 3-connected graphs.) An equivalent way to state this is that two 
graphs have isomorphic lattices of integer flows if and only if 
they have the isomorphic graphic matroids. Existing proofs of this fact are 
not constructive: they don't reconstruct the graphic matroid from the 
lattice of integer flows, but prove that the latter determines the former 
up to isomorphism. In this paper we give an algorithm 
which takes the Gram matrix of $\calF(G)$ as input and produces the graphic 
matroid $\calM(G)$ as output. This, together with existing algorithms
that reconstruct a representative the 2-isomorphism class of $G$ from 
$\calM(G)$~\cite{Fuji, BixbyWagner} fulfills the goal of the paper.

We note that in \cite{SuWagner} Su and Wagner prove the more general 
statement that the lattice of integer flows of a regular matroid determines 
the matroid up to co-loops. (Note that a graph is 2-edge-connected if and 
only if its graphic matroid has no co-loops.) Their proof is partially 
constructive, but requires the input of a
{\em ``fundamental basis of $\calF(M)$ coordinatized by a basis of 
$M$''}. In graphic matroids, this translates to 
a {\em spanning-tree basis of $\calF(G)$} (as defined in Section 
\ref{sec.prelim}). Although we don't build on \cite{SuWagner}, in 
Remark~\ref{rmk.SpanTree} we show how one can choose a spanning-tree basis, 
making the \cite{SuWagner} result fully constructive for graphic matroids.

The concept of 2-isomorphism is not only important in graph theory, but has 
deep connections to other areas of mathematics. For example, it is the 
graph-theoretic equivalent of the Conway mutation of links in low-dimensional 
topology, and this fact has strong applications in knot theory \cite{Greene}. 
In fact, in \cite{Greene} Greene proves that even the {\em $d$-invariant} of 
the lattice $\calF(G)$ is a complete invariant of 2-edge-connected graphs 
up to 2-isomorphism~\cite[Thm 1.3]{Greene}.

The construction of the graphic matroid $\calM(G)$ of $G$ 
from $\calF(G)$ is stated in Theorems~\ref{thm.3conn} and 
\ref{thm.2conn} for 3-connected and 2-connected graphs respectively. 
Our method uses the relationship between the cycle structure of $G$ and the 
geometry of the Voronoi cell $V(G)$ of $\calF(G)$, and relies on  
on Amini's Theorem \cite{Amini} which asserts that the poset of faces of $V(G)$
is isomorphic to the poset of {\em stronly connected orientations} of $G$. 

Since $\calF(G)$ is a complete 2-isomorphism invariant of 2-edge-connected 
graphs $G$, any
other 2-isomorphism invariant (that is, any graph invariant that is 
preserved by 2-isomorphisms) 
is determined by $\calF(G)$. All 2-isomorphism invariants are (often openly, 
sometimes secretly)
matroid invariants, and hence can be computed from the graphic matroid. 
Hence, this paper provides an algorithmic way of computing any of these 
invariants from $\calF(G)$: examples include the number of edges or vertices, 
the Tutte polynomial, the Jones polynomial of the corresponding alternating 
link, and Wagner's algebra of flows \cite{Wagner:FlowsOnGraphs}. 
In addition, as we mentioned earlier, a representative for the 2-isomorphism 
class of $G$ can be found by an application of known graph realization 
algorithms \cite{Fuji, BixbyWagner}. These are quite fast: \cite{BixbyWagner} 
is almost linear time in the number of non-zero entries of the matrix 
representing $\calM(G)$.


\section{Definitions and classical results}
\label{sec.prelim}

Let $G$ be a finite, connected graph with vertex set $V(G)$ and edge set $E(G)$.
Throughout this paper multiple edges and loops are allowed. We use the word 
{\em subgraph} to mean a {\em spanning subgraph} of $G$, that is, a graph $H$ 
such that $V(H) = V(G)$ and $E(H) \subset E(G)$ (we allow the
possibility $E(H)=E(G)$). In other words, a subgraph is
determined by the choice of a subset of the edge set.

We say that $G$ is {\em 2-edge-connected} if it remains connected 
after removing any one edge.
In this paper we will short this to {\em 2-connected}.
Note that elsewhere in the literature the shorthand 2-connected usually means 
{\em 2-vertex-connected}, that is, a graph that remains connected upon
deleting any one vertex along with its incident edges. Here we only 
talk about edge connectedness
unless otherwise stated, hence our choice of terminology.

If $G$ is not 2-connected, then it has an edge $e$ such that the graph 
resulting from deleting $e$, 
$G-e$, is no longer connected. We call such an edge $e$ a {\em bridge}. 
In other words, a graph is 2-connected if and only if it
does not contain a bridge. It is also easy to see that a graph $G$ is 
2-connected if and only if each
edge of $G$ participates in a {\em circuit}. A circuit is a cycle 
(closed walk) in $G$ which does not
go through the same edge or vertex twice. 

Similarly, a connected graph $G$ is called {\em 3-connected} if it remains 
connected upon removing any pair of 
edges. A pair of edges $\{e,f\}$ such that $G-\{e,f\}$ is no longer 
connected is called a {\em 2-cut}.
A graph $G$ is 3-connected if and only if it does not contain a 2-cut.

Let $G$ be a 2-connected graph with edge set $E(G)$, and choose an 
{\em orientation} $\omega$ for $G$: that is, a direction for each edge of $G$. 
This makes $G$ into a one-dimensional CW-complex, and 
{\em integer flows} on $G$ are the elements of its first 
homology, denoted $\calF(G)$.
In other words, for each vertex $v$ in the vertex set $V(G)$, let 
$E_v^+$ and $E_v^-$ denote the set of incoming and outgoing edges at $v$, 
respectively. An element $\sum_{e\in E(G)} a_e e$ of $\mathbb Z^{E(G)}$ 
(the free abelian group generated by $E(G)$) is an 
integer flow if and only if $\sum_{e\in E_v^+} a_e -\sum_{e \in E_v^-} a_e=0$ for 
every $v \in V(G)$. Elements of $\calF(G)$ are called the {\em cycles} of $G$.
Note that $\calF(G)$ is a free abelian group; its rank is called the 
{\em genus} of $G$.

A {\em lattice} is a free abelian group with a non-degenerate inner product.
The set of integer flows $\calF(G)$ admits a lattice structure thanks to the 
inner product induced by the Euclidean inner product on $\mathbb Z^ {E(G)}$,
given by $(e,f)=\delta_{ef}$ for $e, f \in E(G)$.
Note that the isomorphism class of the lattice $\calF(G)$ does not depend on 
the choice of the orientation $\omega$. 
 
An {\em oriented circuit} of $G$ is a closed walk in the (unoriented) graph $G$ 
which does not return to the same vertex or edge twice.
Equivalently, an oriented circuit is a minimal 2-regular subgraph 
of $G$ with a chosen cyclic orientation. 
Each oriented circuit $C$ of $G$ gives rise to an element $C \in \calF(G)$ 
(by a slight abuse of notation),
where the coefficient of $e\in E(G)$ is 0 for edges not in $C$, 1 for those 
edges of $C$ whose orientations are respected, and $-1$ on the edges 
of $C$ which are used ``backwards'' in $C$. The reverse circuit 
induces the element $-C \in \calF(G)$. These are called {\em circuit elements}
of $\calF(G)$

A basis of $\cal F(G)$ is called a {\em circuit basis} if it consists of 
circuit elements only. Choosing a
spanning tree $T$ for $G$ determines a circuit basis for $\calF(G)$, 
consisting of the {\em fundamental circuits} of each
edge of $G$ not in $T$. Namely, each edge $e \notin T$ induces a unique 
circuit consisting of $e$ and the unique path in $T$ from the end of $e$
to its beginning: this is the fundamental circuit of $e$. Such a circuit
basis is called a {\em spanning tree basis}. Note that there exist circuit 
bases that are not spanning tree bases.

Two graphs are {\em isomorphic} if there is an edge-preserving bijection
of their sets of vertices. Two graphs are {\em 2-isomorphic} if there is a 
cycle preserving bijection of their sets of edges.

A connected graph is {\em 3-vertex-connected} if it remains connected 
upon removing any three
vertices. In \cite{Whitney} Whitney showed
that two 3-vertex-connected graphs are 2-isomorphic if and only if they are 
isomorphic. Furthermore, Whitney gave a set of local moves that generate the
2-isomorphism of 2-connected graphs. For a detailed discussion see for 
example \cite{Greene}.

If two graphs are 2-isomorphic, then it is easy to see that their
lattices of integer flows are isomorphic. If $G$ and $G'$
graphs with orientations $\omega$ and $\omega'$, respectively. 
Then a 2-isomorphism $\bar{\varphi}: E(G) \to E(G')$ can be lifted to an 
isomorphism $\tilde{\varphi}: \mathbb Z^{E(G)} \to \mathbb Z^{E(G')}$, for 
which $\varphi(e) = \pm \bar{\varphi}(e)$, and which
restricts to an isomorphism $\varphi: \calF(G)\to \calF(G')$. To see this, 
write the two-isomorphism as a series of Whitney-moves, this way $\omega$ 
induces an orientation on $G'$. Choose the sign of $\bar{\varphi}(e)$ to
be positive when this induced orientation agrees with $\omega'$ and negative 
otherwise. Conversely, by the results of \cite{Caporaso} and \cite{SuWagner}, 
if two graphs have isomorphic lattices of integer flows, then they are 
2-isomorphic. In fact, an even stronger statement is true, which follows 
from the proof of \cite[Thm.3.8]{Greene}.

\begin{theorem}[{\cite{Greene}}]
\label{thm.InducedIso}
Let $G$ and $G'$ be graphs with orientations $\omega$ and $\omega'$, 
respectively. Let $\varphi: \calF(G) \to \calF(G')$ be an isomorphism of 
their lattices of integer flows. Then $\varphi$ can be extended to an 
isomorphism $\tilde{\varphi}: \mathbb Z^{E(G)} \to \mathbb Z^{E(G')}$,
satisfying that $\tilde{\varphi}|_{\calF(G)}=\varphi$ and for every 
$e \in E(G)$, $\tilde{\varphi}(e)=\pm e'$ for some $e'\in E(G')$. In other 
words, $\tilde{\varphi}$ is a lift of a 2-isomorphism.
\end{theorem}

An equivalent definition of 2-isomorphism is to say that a graph 
2-isomorphism is an isomorphism of the corresponding graphic matroids. 
A {\em matroid} is a common generalization of the notion of linear 
independence and graphs up to 2-isomorphism. Matroids have several 
equivalent definitions, for a detailed introduction see \cite{Oxley}.
One definition of a matroid is as a pair consisting of a {\em ground set} 
and a set of subsets which are called the {\em independent sets}. The set 
of independent sets is subject to the expected axioms: the empty set is 
independent, subsets of independent sets are independent, and the 
``exchange property''. Alternatively, one can declare the set of 
{\em dependent subsets} of the ground set, subject to different axioms. 
Yet another equivalent definition is to specify
the set of minimal dependent 
subsets or {\em circuits}: those whose proper subsets are all independent.

In this paper it will be most convenient to work with the circuit definition. 
That is, the graphic matroid $\calM(G)$ is given by the sets of edges and 
circuits of $G$, along with the data of which edges participate in any 
given circuit. In other words, $\calM(G)$ is a $(0,1)$-matrix whose columns 
are indexed by $E(G)$ and whose rows are indexed by the circuits of $G$. 
Having an entry 1 in position $(e,C)$ means that the edge $e$ participates 
in the circuit $C$.

The first step to reconstructing $\calM(G)$ is to choose a circuit basis 
for $\calF(G)$. To do this we have to be able to recognize the circuit 
elements in $\calF(G)$ without knowledge of $G$. Note that once this 
is done, we at least know how many edges participate in each circuit,
as the length of a circuit $C$ is then given by the norm $(C,C)$ in 
$\calF(G)$. Also, since $G$ is 2-connected, all edges of $G$ participate in 
some circuit, and hence ``appear'' in some element of the circuit basis. 

One detects circuit elements of $\calF(G)$ by studying the 
{\em Voronoi cell} $V(G)$ of the lattice, according to a theorem of 
Bacher, de la Harpe, and Nagnibeda \cite{BDN:CutsAndFlows}. The Voronoi cell 
is the collection
of points in $\calF(G)\otimes \bbR$ which are closer to the origin than to 
any of the other lattice points. This is a polytope, whose codimension one
faces are supported by the perpendicular bisector hyperplanes of some 
lattice vectors. As it turns out, the oriented circuits of $G$ are in 
bijection with the codimension one faces of $V(G)$.

\begin{proposition}[{\cite[Prop.6]{BDN:CutsAndFlows}}]
\label{prop.Voronoi}
The following are equivalent:
\begin{itemize}
\item The lattice vector $v\in \calF(G)$ is a circuit element.
 \item The perpendicular bisector of $v$, that is, the hyperplane given by 
$(x,v)=\frac{(v,v)}{2}$, supports a codimension one face of the Voronoi cell.
 \item The vector $v$ is a {\em strict Voronoi vector}, that is, $\pm v$ are 
the two unique shortest elements of the coset $v+2\calF(G)$.
\end{itemize}
\end{proposition}

\begin{example}
\label{ex.TwoCycles}
As an example consider the graph $H$ shown in Figure \ref{fig.ExampleGraph}. 
The genus of $H$ (rank of $\calF(H)$) is 2, and a basis for $\calF(H)$ 
consists of the circuits $C_1$ and $C_2$. Their inner products are given by 
$(C_1,C_1)=(C_2,C_2)=3$ and $(C_1,C_2)=-1$.
Hence, the lattice of integer flows can be represented isometrically in the 
Euclidean plane by vectors $C_1=(3,0)$ and $C_2=(-1/3, 2\sqrt{20}/3)$, as
shown in Figure \ref{fig.ExampleGraph}. 
The Voronoi cell is a hexagon supported by the perpendicular bisectors of 
$\pm C_1$, $\pm C_2$ and $\pm (C_1+C_2)$, 
which are exactly the circuits of $H$.
\end{example}

\begin{figure}[!htpb]
\raisebox{3mm}{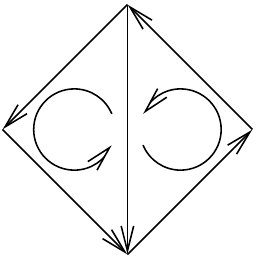} \hspace{2cm} 
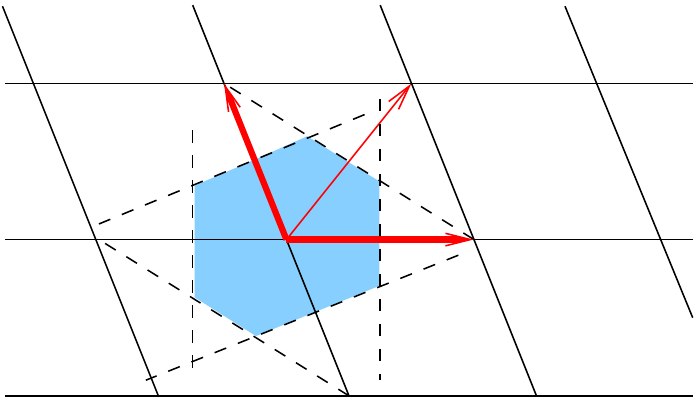
\caption{The graph $H$ and the Voronoi cell of its lattice of integer flows.}
\label{fig.ExampleGraph}
\end{figure}


\section{Reconstructing $\calM(G)$ from $\calF(G)$}
\label{sec:main}

\subsection{General discussion}

Choose a circuit basis for $\calF(G)$; the length of each circuit $C$ in 
the basis is the norm $(C,C)$ in $\calF(G)$. In order to reconstruct 
$\calM(G)$, we need more information about the edges on $G$. 
We recover this information using a result by Amini \cite{Amini}
which states that the poset of faces of the Voronoi cell of $\calF(G)$ 
(relative to inclusion) is isomorphic to the 
{\em poset of strongly connected orientations} of subgraphs of $G$. In fact,
much of this section amounts to developing an explicit understanding of 
Amini's result.

We start by recalling the definition of strongly connected orientations. 
Let $\omega$ be an orientation of a connected graph $G$. Then $\omega$ is 
called {\it strongly connected}, or a {\it strong orientation}, if for any
pair of vertices $u$ and $v$ there is an oriented path from $u$ to $v$ as 
well as an oriented path from $v$ to $u$. That is, any two vertices $u$ and 
$v$ participate in an oriented cycle which does not repeat edges. Note that 
in particular this implies 
that $G$ is 2-connected. If $G$ is a not necessarily connected graph, we
say that an orientation is strongly connected if it is strongly connected 
on all of the connected components of $G$. A classical theorem of Robbins
\cite[Chapter 12.1]{GrossYellen} states that 
a (not necessarily connected) graph $G$ has a strong orientation if and only 
if all of its connected components are 2-connected.

We recall the definition of the poset $\calS(G)$ of strongly connected 
orientations \cite{Amini}. The elements of $\calS(G)$ are pairs 
$(H,\omega_H)$, where $H$ is a subgraph of $G$ and $\omega_H$ is a 
strongly connected orientation of $H$.
A partial order $\preceq$ is defined by $(H,\omega_H)\preceq(H',\omega'_{H'})$ 
if $H'$ is a subgraph of $H$, and $\omega'_{H'}$ is a restriction of 
$\omega_H$ to $H'$. Observe that this poset has a unique maximal element 
$(\emptyset, \emptyset)$.

What are the sub-maximal elements of $\calS(G)$? These are minimal strongly 
oriented, and hence 2-connected, subgraphs of $G$. A minimal 2-connected 
subgraph is a circuit, and each circuit has two strongly connected 
orientations (the cyclic ones). Hence, Proposition~\ref{prop.Voronoi} 
implies that the sub-maximal elements of $\calS(G)$ are in bijection with 
the codimension one faces of the Voronoi cell of $\calF(G)$. In fact,
more is true.

\begin{theorem}[\cite{Amini}]
\label{thm.Amini}
Let $\calV(G)$ be the poset of faces of the Voronoi cell $V(G)$ of $\calF(G)$, 
with the partial order given by inclusion. Then the posets $\calV(G)$ and 
$\calS(G)$ are isomorphic, and the codimension of a face of $V(G)$ is equal 
to the genus of the subgraph corresponding to it under the isomorphism. 
\end{theorem}

\begin{example}
\label{ex.3D1}
Consider the graph $G$ shown in Figure \ref{fig.3triangles}.
Let us choose a circuit basis 
$$
C_1=e_3+e_4+e_7, \qquad C_2=e_5+e_6+e_7, \qquad C_3=e_1+e_2+e_3+e_4.
$$ 
The Gram matrix $(C_i,C_j)$ of $\calF(G)$ in this basis is
$$
\begin{pmatrix}
3 & 1 & 2  \\
1 & 3 & 0  \\
2 & 0 & 4  
\end{pmatrix}
$$

\begin{figure}
$$
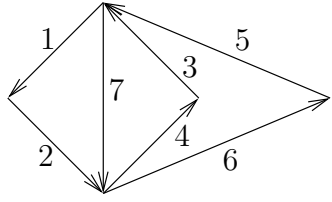
$$
\caption{A graph with $5$ vertices, $7$ edges $e_1,\dots,e_7$ with
orientation $\omega$}
\label{fig.3triangles}
\end{figure}

\begin{figure}
\includegraphics[height=4cm]{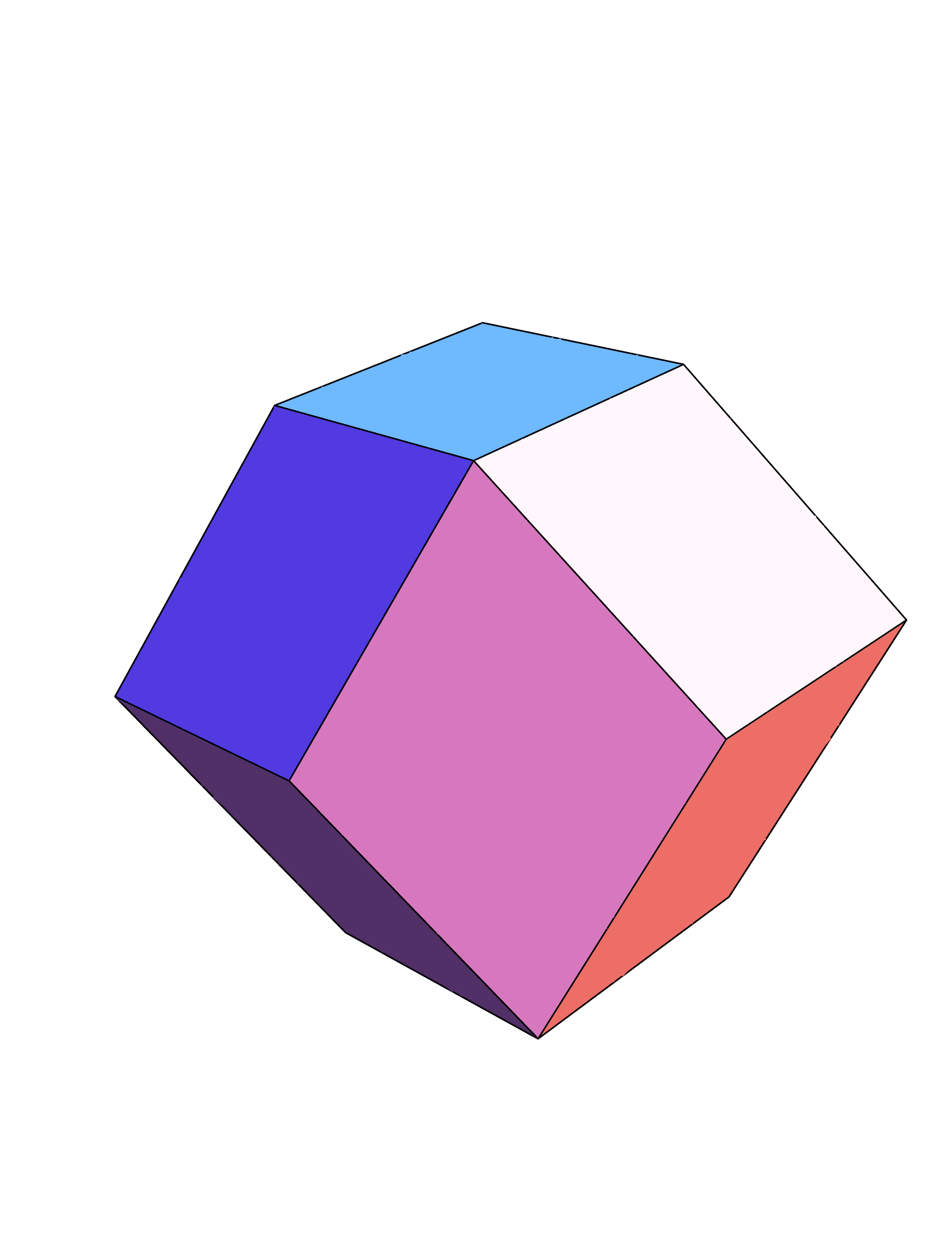}
\caption{A rhombic dodecahedron}
\label{fig.rhombicdodecahedron}
\end{figure}

The Voronoi cell of $\calF(G)$ is a {\em rhombic dodecahedron} (shown in
Figure \ref{fig.rhombicdodecahedron}) with $14$ vertices (in bijection with 
the strong orientations of $G$), $24$ edges, and $12$ faces (in bijection 
with the oriented circuits of $G$). All faces are rhombi: each oriented 
circuit is contained in four different strongly oriented genus-2 subgraphs. 
Every vertex has valency $3$ or $4$. The orientation $\omega$ of 
Figure~\ref{fig.3triangles} is in fact a strong orientation,
and it corresponds to a 4-valent vertex of the Voronoi cell, as there are
four $\omega$-compatibly oriented circuits given by
$$
C_1=e_3+e_4+e_7, \qquad C_2=e_5+e_6+e_7, \qquad C_3=e_1+e_2+e_3+e_4,
\qquad C_4=e_1+e_2+e_5+e_6. 
$$
\end{example}

An immediate corollary of Theorem~\ref{thm.Amini} is the following:

\begin{corollary}
\label{cor.UnionOfCircuits}
Let $F$ be a face of $\calV(G)$ expressed
as an intersection $$F=F_{C_1}\cap F_{C_2} \cap \dots \cap F_{C_s},$$ where 
$\{F_{C_i}: i=1 \dots s\}$ are codimension one faces corresponding to oriented 
circuits $C_1, C_2, \dots C_s$, respectively. Then the orientations of these 
circuits are compatible, and the face $F$ corresponds under Amini's 
isomorphism to the subgraph $C_1 \cup C_2 \cup \dots \cup C_s$ 
with the induced orientation. \qed
\end{corollary}

This observation lets us choose a better circuit basis: one where the sizes 
of intersections of pairs of basis circuits are computed by their pairings: 

\begin{definition}
A circuit basis is {\em cancellation-free} if, for any pair of basis 
circuits $C_i$ and $C_j$, the number of edges in the intersection of the 
two circuits is computed by their pairing: $|C_i\cap C_j|=|(C_i,C_j)|$.
\end{definition}

Note that spanning tree bases are always cancellation-free, as the 
intersection of any two basis circuits is a simple path. This in particular 
implies that cancellation-free circuit bases exist for any graph. The 
following remark shows how to choose a cancellation-free circuit basis 
based on $\calF(G)$ alone, without knowledge of $G$ or a spanning tree of $G$.

\begin{remark}
\label{rmk.GoodBasis}
For any pair of circuits $C_i$ and $C_j$, there are three possibilities:
\begin{enumerate}
\item 
The orientations of $C_i$ and $C_j$ are compatible. This is the case if and 
only if the corresponding faces $F_{C_i}$ and $F_{C_j}$ intersect.
\item 
The orientations of $C_i$ and $-C_j$ are compatible, as in Figure 
\ref{fig.ExampleGraph} for example. This is the case if and only if $F_{C_i}$ 
and $F_{-C_j}$ intersect. 

\noindent
(Note that for any circuit $C$ and corresponding 
face $F_C$, $F_{-C}=-F_C$. Here $-F_C$ the face parallel to and geometrically 
opposite from $F_{C}$. Note also that (1) and (2) co-occur if and only if 
the circuits $C_i$ and $C_j$ are edge-disjoint.)
\item 
Neither of the above is true, that is, $C_i$ and $C_j$ cannot be compatibly 
oriented. An example is given by the circuits $C_1$ and $C_2$ in 
Figure \ref{fig.BadCircuits}. This is the case if and only if the faces 
$\{F_{C_i}, F_{C_j}, F_{-C_i}, F_{-C_j}\}$ are pairwise disjoint.
\end{enumerate}
In the first two cases $|C_i \cap C_j|=|(C_i,C_j)|$, but not in the
third case. Hence, to choose a cancellation-free basis, one only needs to
make sure that all pairs of basis circuits are of the first two types, 
and this can be detected by examining the one-codimension faces of $\calV(G)$.
In Remark~\ref{rmk.SpanTree} we will also show how to choose a spanning 
tree basis, but this takes more work and is not necessary in order to 
compute $\calM(G)$.
\end{remark}

\begin{figure}[!htpb]
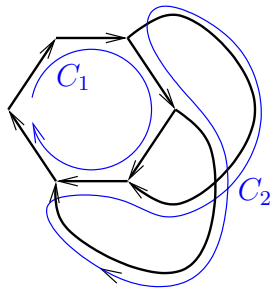
\caption{Two circuits where no compatible orientation is possible.}
\label{fig.BadCircuits}
\end{figure}

To continue understanding Theorem~\ref{thm.Amini} explicitly, we analyze 
which set of faces of the Voronoi cell $V(G)$ correspond to a given strongly 
orientable subgraph $H$ with different strong orientations.

\begin{proposition}
\label{prop.parallel}
Let $H$ be a strongly orientable subgraph of $G$. Let $\omega_H$ be a strong 
orientation 
of $H$ and let $F_{H,\omega_H}$ be the face of $V(G)$ corresponding to the pair 
$(H, \omega_H)$. Then the set of faces corresponding to $H$ with all of its 
different strong orientations is the set of faces parallel to, and of the 
same codimension as $F_{H,\omega_H}$. 
\end{proposition}

\begin{proof}
If $F_{H,\omega_H}$ is of codimension $s$, then it can be written (not 
necessarily uniquely) as an intersection 
$F_{H,\omega_H}=F_{C_1}\cap \dots \cap F_{C_s}$, where $F_{C_1}, \dots F_{C_s}$ 
are codimension one faces corresponding to oriented circuits $C_1,\dots C_s$.
Then, by Corollary~\ref{cor.UnionOfCircuits}, $H$ can be written as 
$H=C_1\cup \dots \cup C_s$. Consider another codimension $s$ face $F'$, and 
write $F'=F_{C'_1} \cap \dots \cap F_{C'_s}$ as the intersection of $s$ 
codimension one faces. 
 
The ``direction'' of $F_{H,\omega_H}$ is the orthogonal complementary subspace 
to the span of the lattice vectors $\{C_1,\dots,C_s\}$, and similarly for 
$F'$. So $F'$ and $F$ are parallel if and only if 
$\text{span}\{C'_1, \dots, C'_s\}=\text{span}\{C_1,\dots ,C_s\}$. 
 
Note that $\text{span}\{C_1,\dots ,C_s\}$ is in fact the cycle space of $H$, 
that is, the subspace of $\calF(G) \otimes \mathbb R$ containing the cycles 
in $H$ (i.e., the first homology of $H$). Similarly, 
$\text{span}\{C'_1, \dots, C'_s\}$ is the cycle space of $H'$. Since all 
components of $H$ and $H'$ are two-connected, every edge in $H$ or $H'$
participates in a circuit in $H$ or $H'$, respectively. Hence, cycle spaces 
of $H$ and $H'$ agree if and only if $H'=H$, completing the proof.
\end{proof}

\subsection{3-Connected graphs}
Our next task is to understand the edges (1-dimensional faces) of the Voronoi 
cell, which is easily done for 3-connected graphs, and leads to
an immediate construction of $\calM(G)$ in this case. 
By Amini's Theorem the edges of $\calV(G)$ correspond to pairs $(H,\omega_H)$, 
where $H$ 
is a maximal strongly orientable proper subgraph of $G$ and $\omega_H$ is a 
strong orientation of $H$. If $G$ is 3-connected, then for any edge 
$e\in E(G)$, $G-e$ is 2-connected and hence strongly orientable, and all 
maximal strongly orientable proper subgraphs are of this form. By 
Proposition \ref{prop.parallel}, there is a {\em parallel class} of edges 
of $V(G)$ corresponding to the subgraph $G-e$ with different strong 
orientations. This results in the following construction of the 
graphic matroid $\calM(G)$ from $\calF(G)$ for 3-connected graphs $G$.

\begin{theorem}
\label{thm.3conn}
Let $G$ be a 3-connected graph with lattice of integer flows $\calF(G)$, 
whose Voronoi cell is denoted $V(G)$. 
Then 
\begin{enumerate} 
\item 
There is a bijection between $E(G)$ (that is, the ground set of $\calM(G)$) 
and the parallel classes of edges of the Voronoi cell $V(G)$.
\item 
There is a bijection between the circuits of $G$ (that is, the circuits of 
$\calM(G)$) and the parallel classes of codimension one faces of $V(G)$.
\item 
A given edge $e\in E(G)$ belongs to a circuit $C$, that is 
$C \not\subset G-e$,
if and only if no member of the parallel class of 
edges of $V(G)$ corresponding to $e$ belongs to the face $F_C$ (and 
consequently none of them belong to $F_{-C}$ either). 
\end{enumerate}
\end{theorem}

\subsection{2-Connected graphs}

The situation is more complicated when $G$ is not 3-connected. Edges of the 
Voronoi cell still correspond to pairs $(H, \omega_H)$, where $H$ is a maximal 
strongly orientable proper subgraph of $G$. The subgraph $H$ may be of the 
form $G-e$ for some edge $e$, when $e$ does not participate in a 2-cut and 
hence $G-e$ is 2-connected, thus strongly orientable. 

It is also possible that $H=G-S$, where $S$ is a set of multiple edges: all 
of these edges must participate in 2-cuts, otherwise $H$ would not be maximal. 
Note that in this case $H$ has several connected components: indeed,
if an edge $e\in S$ participates in a 2-cut $\{e,f\}$ in $g$, then $f$ is a 
bridge in $G-e$, and so $f\notin H$ as $H$ is strongly orientable.

Constructing $\calM(G)$ requires understanding the complements $S$ in $E(G)$ 
of maximal strongly orientable proper subgraphs. They turn out to 
be what we will call {\em 2-cut blocks}, so we begin by defining and 
describing these.

Lemmas~\ref{lem.EqRel},~\ref{lem.2CB} and ~\ref{lem.2CBCircuit} could be 
deduced using classical results on minimally 2-connected graphs (see
for example \cite[Chapter 1.3]{Bol}), but we choose to prove them directly as 
the proofs are short and elementary.

\begin{lemma}
\label{lem.EqRel}
The relation ``edges $e$ and $f$ are either the same or form a 2-cut 
in $G$'' is an equivalence relation on $E(G)$.
\end{lemma}

\begin{proof}
The reflexive and symmetric properties are obvious, so the only point to prove 
is transitivity. Namely, we need to show that if $\{e, f, h\}$ are three 
distinct edges of $G$ and $\{e,f\}$ and $\{f,h\}$ are both 2-cuts in $G$ 
then so is $\{e,h\}$.
 
Indeed, in the graph $G-f$, the edges $e$ and $h$ are both bridges, so the 
graph $G-\{e,f,h\}$ has three connected components. Of these three $f$ can 
only join two, so the graph $G-\{e,h\}$ is still disconnected, meaning that
$\{e,h\}$ is a 2-cut.
\end{proof}

\begin{definition}
\label{def.2CB}
We call the equivalence classes of the above relation {\em 2-cut blocks} 
in $E(G)$. 
\end{definition}

\begin{lemma}
\label{lem.2CB}
If $S \subset E(G)$, then the subgraph $H=G-S$ is a maximal proper 
strongly orientable subgraph of $G$ if and only if $S$ is a 2-cut block.
\end{lemma}

\begin{proof}
Observe that if $H$ is strongly orientable then $S$ is a union of 2-cut 
blocks. Indeed, if $H$ is strongly orientable then all of its connected 
components are two-connected, and hence if for some edges $e,f \in E(G)$, 
$e\in S$ and $\{e,f\}$ is a 2-cut, then it must be that $f \in S$ as well. 
On the other hand when $S$ is a 2-cut block then $G-S$ is strongly orientable 
as it contains no bridges: if $e$ is a bridge in $G-S$, then $e$ 
participates in a 2-cut $\{e,f\}$ in $G$ with $f \in S$, contradicting that 
$S$ is a 2-cut block. Hence, if $G-S$ is maximal then $S$ is a single 2-cut 
block and vice versa, as needed. 
\end{proof}

For each edge $\epsilon$ of the Voronoi cell $V(G)$, let $[\epsilon]$ denote 
the set of all the edges of $V(G)$ parallel to $\epsilon$, that is, the 
parallel class of $\epsilon$. By Proposition~\ref{prop.parallel}, the 
parallel classes $[\epsilon]$ are in one-to-one correspondence with maximal 
2-connected proper subgraphs $H=G-S$, where $S$ is a 2-cut block. Recall that 
$S$ is a single edge $\{e\}$ if and only if $G-e$ is two-connected,
that is if and only if $e$ does not participate in any 2-cuts. 
Let $S_{[\epsilon]}$ be the 2-cut block corresponding to the parallel class 
$[\epsilon]$.

The construction of $\calM(G)$ comes down to detecting the 
size of $S_{[\epsilon]}$ for each $[\epsilon]$. As in the case of 3-connected 
graphs in Theorem~\ref{thm.3conn}, one can then list the circuits 
of $G$ and tell which 2-cut blocks participate in each circuit based on the 
poset of faces of $V(G)$. Note that circuits are always unions of entire 
two-cut blocks, as the following Lemma shows.

\begin{lemma}
\label{lem.2CBCircuit}
If $S$ is a 2-cut block and $C$ is a circuit in $G$, then either $S$ and 
$C$ are disjoint or $S\subset C$.
\end{lemma}

\begin{proof}
Assume that $S\cap C \neq \emptyset$ and let $e\in S\cap C$ be some edge in 
the intersection. Let $f \in S$ be any other edge in $S$. We need to show 
that $f \in C$. 
 
Since $f \in S$, we know that $\{e,f\}$ is a 2-cut in $G$, hence $e$ is a 
bridge in $G-f$, so $e$ does not participate in any circuit in $G-f$. So 
$C \not\subset (G-f)$, therefore $f \in C$.
\end{proof}

Let $n_{[\epsilon]}:=|S_{[\epsilon]}|$ denote the cardinality of the two-cut 
block $S_{[\epsilon]}$. To finish the construct $\calM(G)$, we need to compute 
the numbers $n_{[\epsilon]}$. Let $\calB=\{C_1,...,C_r\}$ be a 
cancellation-free circuit basis of $\calF(G)$ chosen as in 
Remark \ref{rmk.GoodBasis}. For each $S_{[\epsilon]}$ there is some circuit 
$C_i\in \calB$ for which $S_{[\epsilon]} \subset C_i$. Each $C_i\in \calB$ 
corresponds to a
codimension one face of $V(G)$, denoted $F_{C_i}$. For each $C_i \in \calB$, 
the edges of $C_i$ can be expressed as a disjoint union of the sets 
$S_{[\epsilon]}$, where $[\epsilon]$ runs over all edge directions of the 
Voronoi cell which {\em do not} participate in the face $F_{C_i}$,
as in Theorem~\ref{thm.3conn}. Hence, for $i=1,\dots,r$ one can write 
the linear equations 
\begin{equation}
\label{eq.CycLength}
\sum_{[\epsilon]\notin F_{C_i}}n_{[\epsilon]}= (C_i,C_i),
\end{equation}
where the notation $[\epsilon] \notin F_{C_i}$ means that no edge parallel 
to $\epsilon$ participates in the face $F_{C_i}$. In addition we have a 
similar linear equation for each pair of circuits $C_i, C_j \in \calB$, 
as the number of edges in $C_i\cap C_j$ is given by the absolute value of 
their pairing $|(C_i, C_j)|$, as $\calB$ is cancellation-free.
\begin{equation}
\label{eq.CycIntersect}
\sum_{[\epsilon]\notin F_{C_i},F_{C_j}}n_{[\epsilon]}= |(C_i,C_j)|.
\end{equation}
Note that the resulting system of linear equations always has a positive 
integer solution, namely, the numbers $n_{[\epsilon]}$ determined by the 
graph $G$.

\begin{proposition}
The system~\eqref{eq.CycLength},~\eqref{eq.CycIntersect} has a unique 
positive integer solution $\{n_{[\epsilon]}\}$.
\end{proposition}

\begin{proof}
Assume that $\{n_\epsilon\}$ is the solution given by the sizes of 
2-cut-blocks in $G$, and $\{n'_{[\epsilon]}\}$ is a different positive integer 
solution. Our strategy is to construct a corresponding graph $G'$ 
with 2-cut blocks $S'_{[\epsilon]}$ of sizes $n'_{[\epsilon]}$, and show that 
$G$ and $G'$ are 2-isomorphic. 

The edges of $G$ are paritioned by the sets $S_{[\epsilon]}$, with 
$|S_{[\epsilon]}|=n_{[\epsilon]}$. Assume that for some 
$\epsilon$, $n'_{[\epsilon]}>n_{[\epsilon]}$. Choose an arbitrary edge of 
$S_{[\epsilon]}$ and split it into $n'_{[\epsilon]}-n_{[\epsilon]}+1$ edges
by creating $n'_{[\epsilon]}-n_{[\epsilon]}$ degree two vertices on it, 
preserving the edge orientation. Call the 
resulting graph $G^*$, and the changed 2-cut block $S^*_{[\epsilon]}$. 
First, note that in $G^*$, $S^*_{[\epsilon]}$ is a in fact 2-cut block, that 
is, any two of its edges form a 2-cut. Furthermore, 
$|S^*_{[\epsilon]}|=n'_{[\epsilon]}$, and all other 2-cut blocks are unchanged. 
In addition, the circuits of $G$ are in bijection with the 
circuits of $G^*$, in the obvious way, and 
linear independence of circuits is preserved. (Informally speaking, the edge 
splitting operation does not change the cycle structure of $G$ except for the 
lengths of cycles.)

On the other hand, if for some $\epsilon$, $n'_{[\epsilon]}<n_{[\epsilon]}$, 
then choose $n_{[\epsilon]}-n'_{[\epsilon]}$ arbitrary edges of $S_{[\epsilon]}$
and contract them. Again call the resulting graph $G^*$, and note that the 
newly created $S^*_{[\epsilon]}$ is a 2-cut block in $G^*$ of cardinality 
$n'_{[\epsilon]}$, while all other 2-cut blocks are unchanged. 

This operation also preserves the circuits of $G$: first note 
that any circuit in $G$ which intersects with $S_{[\epsilon]}$ contains it,
and $n'_{[\epsilon]}$ is still a positive integer, so circuits are never 
contracted to nothing. Furthermore, circuits remain circuits: 
only edges that participate in 2-cuts are ever contracted and hence, if
a closed walk didn't repeat vertices in $G$, it still does not do so in $G^*$.
Finally, the edge contractions do not create new circuits: if a closed walk 
in $G$ repeats a vertex, it will still do so in $G^*$.
In summary, the circuits of $G$ are in bijection with the circuits 
of $G^*$. Linear independence of circuits is also preserved.

Now carry out this process for all $[\epsilon]$ to create a new graph $G'$ 
with 2-cut blocks $S'_{[\epsilon]}$ of sizes $n'_{[\epsilon]}$. We claim that 
$\calF(G) \cong \calF(G')$. Recall that $\{C_1,\dots,C_r\}$ is a circuit 
basis for $\calF(G)$, and call the gram matrix corresponding to this 
circuit basis $M$. Let $\{C'_1,\dots,C'_r\}$ be the circuits in $G'$ created 
from $\{C_1,\dots,C_r\}$, respectively. Since circuits and linear independence 
were preserved, these form a circuit basis for $\calF(G')$; let us
call the corresponding Gram matrix $M'$. 

Since the $C_i$ form a cancellation-free basis of $\calF(G)$, the 
entries of the absolute value $|M|$ are the sizes of intersections 
$|C_i \cap C_j|$, with $i,j=1...r$. The basis $\{C_1',...,C_r'\}$ is also 
cancellation-free: the characterization of Remark \ref{rmk.GoodBasis} shows 
that the edge splittings and contractions do not change this property. 
Since the $\{n_{\epsilon}\}$ and $\{n'_{[\epsilon]}\}$ are both 
solutions to the system of equations~(\ref{eq.CycLength}) and 
(\ref{eq.CycIntersect}), we deduce that $|M|=|M'|$.

Furthermore, the sign of the $(i,j)$ entry of $M$ depends on whether $C_i$ 
and $C_j$ are oriented compatibly or opposite (where opposite means that 
the orientation of $C_i$ is compatible with that of $-C_j$). This is 
unchanged by edge splittings and contractions, hence the sign of each entry 
is the same in $M$ and $M'$. So $M=M'$ and $\calF(G) \cong \calF(G')$. Thus, 
by Theorem \ref{thm.InducedIso}, this lifts to an isomorphism 
$\mathbb Z^{E(G)} \to \mathbb Z^{E(G')}$, which sends each edge of $G$ to a 
signed edge of $G'$, creating edge bijections within the 2-cut blocks. 
So $n_\epsilon=n'_\epsilon$ for all $\epsilon$ as needed.
\end{proof}

\begin{theorem}
\label{thm.2conn}
The graphic matroid $\calM(G)$ of a 2-connected graph $G$ can be computed 
explicitly from the lattice of integer flows $\calF(G)$ by the following 
algorithm:
\begin{enumerate} 
\item 
List the one-codimensional faces of the Voronoi cell $V(G)$ of $\calF(G)$, 
and use this to choose a good circuit basis $\{C_1,...,C_r\}$ for $\calF(G)$ 
as explained in Remark~\ref{rmk.GoodBasis}.
\item 
List the edges $\{\epsilon\}$ of $V(G)$, and group them into parallel 
classes $\{[\epsilon]\}$. For each circuit $C$, list which edge directions 
participate in the face $F_{C}$.
\item 
For each basis circuit $C_i, i=1...r$, write the equation 
$$
\sum_{[\epsilon]\notin F_{C_i}} n_{[\epsilon]}=(C_i,C_i),
$$
and for each pair of basis circuits 
$\{\{F_{C_i},F_{C_j}\},i,j=1...r, i \neq j\}$ write the equation 
$$
\sum_{[\epsilon]\notin F_{C_i}, F_{C_j}} n_{[\epsilon]}=(C_i,C_j).
$$
\item 
Solve the system of linear equations, to find the unique positive integer 
solution $\{n_{[\epsilon]}\}$.
\item 
The ground set of $\calM(G)$ is the edge set $E(G)$, which can be written 
as a disjoint union of the sets $\{S_{[\epsilon]}\}$ whose sizes are given 
by the solution $\{n_{[\epsilon]}\}$.
\item 
The edge $e$ belongs to a circuit $C$ if and only if no member of the 
corresponding edge parallel class $[\epsilon]$ in $V(G)$ belongs to the 
face $F_C$. 
\end{enumerate}
\end{theorem}

\begin{example}
\label{ex.3D11}
As an example, let us carry out the algorithm for the graph of 
Example~\ref{ex.3D1}, and construct the graphic matroid $\calM(G)$
from $\calF(G)$. 
\begin{enumerate}
\item 
We start with the Gram matrix of $\calF(G)$ given by
$$
\begin{pmatrix}
3 & 1 & 2  \\
1 & 3 & 0  \\
2 & 0 & 4  
\end{pmatrix}.
$$
(This is the Gram matrix with respect to the circuit basis
in Example~\ref{ex.3D1}.) We compute the Voronoi cell $V(G)$ of $\calF(G)$ 
which is the rhombic dodecahedron shown in Figure~\ref{fig.rhombicdodecahedron}.
It has 12 faces. Of these we choose three pairwise intersecting ones, 
thus forming a cancellation-free basis: let $F_{C_1}$ be the top right 
(white) face of the dodecahedron shown in Figure~\ref{fig.rhombicdodecahedron},
$F_{C_2}$ the top (light blue) face, and $F_{C_3}$ the front (light purple) 
face.
\item 
The rhombic dodecahedron has 24 edges belonging to four parallel classes 
(of 6 edges each). Let $\epsilon_1$ be the edge between the two left side 
faces (dark blue and dark purple), $\epsilon_2$ the edge between the two 
right side faces (white and orange), $\epsilon_3$ be the edge between the 
front and top left faces (light purple and dark blue), and $\epsilon_4$ the 
edge between the front and top right faces (light purple and white).

\noindent
The parallel classes $\{[\epsilon_2],[\epsilon_4]\}$ appear in $F_{C_1}$, 
$\{[\epsilon_1],[\epsilon_2]\}$ appear in $F_{C_2}$, and 
$\{[\epsilon_3],[\epsilon_4]\}$ appear in $F_{C_3}$.
\item 
$n_1+n_3=3, \quad n_3+n_4=3, \quad n_1+n_2=4, \quad n_3=1, \quad n_1=2, 
\quad 0=0$.
\item 
$n_1=2, \quad n_2=2, \quad n_3=1, \quad n_4=2$.
\item 
There are 7 edges partitioned into sets $E(G)=\bigcup_{i=1}^4 S_i$ with 
$|S_i|=n_i$.
\item 
There are six (unoriented) circuits corresponding to parallel pairs of faces 
of $V(G)$. These are: 
$C_1$, $C_2$, $C_3$, $C_4=C_3+C_2-C_1$, $C_5=C_1-C_2$, and $C_6=C_3-C_1$. 
They can be expressed in terms of the $S_i$ as follows: 
$C_1=S_{[\epsilon_1]}\cup S_{[\epsilon_3]}$, 
$C_2=S_{[\epsilon_3]}\cup S_{[\epsilon_4]}$,
$C_3=S_{[\epsilon_1]}\cup S_{[\epsilon_2]}$,
$C_4=S_{[\epsilon_2]}\cup S_{[\epsilon_4]}$,
$C_5=S_{[\epsilon_1]}\cup S_{[\epsilon_4]}$, and
$C_6=S_{[\epsilon_2]}\cup S_{[\epsilon_3]}$.
In other words, $\calM(G)$ is given by the $(0,1)$-matrix
\begin{equation}
\label{eq.MGmatrix}
\begin{bmatrix}
1 & 1 & 0 & 0 & 1 & 0 & 0 \\
0 & 0 & 0 & 0 & 1 & 1 & 1 \\
1 & 1 & 1 & 1 & 0 & 0 & 0 \\
0 & 0 & 1 & 1 & 0 & 1 & 1 \\
1 & 1 & 0 & 0 & 0 & 1 & 1 \\
0 & 0 & 1 & 1 & 1 & 0 & 0 \\
\end{bmatrix}
\end{equation}
\end{enumerate}
In Figure~\ref{fig.3triangles}, we see a representative $G$ of the 
2-isomorphism class of graphs given by $\calF(G)$. In this representative, 
$S_{[\epsilon_1]}=\{e_3,e_4\}$, $S_{[\epsilon_2]}=\{e_1,e_2\}$,
$S_{[\epsilon_3]}=\{e_7\}$, and $S_{[\epsilon_4]}=\{e_5,e_6\}$. 
\end{example}

\begin{remark}
\label{rmk.SpanTree}
In \cite{SuWagner} Su and Wagner prove that the lattice of integer flows 
$\calF(M)$ of a regular matroid $M$ determines the matroid up to co-loops. 
Their proof is almost constructive, but requires the input of a 
``fundamental basis of $\calF(M)$ coordinatized by a basis of $M$''. In the 
context of graphic matroids (which form a sub-class of regular matroids) 
this means the input of a {\em spanning-tree basis}, that is, the set of 
{\em fundamental circuits} corresponding to some spanning tree of $G$, as 
described in Section~\ref{sec.prelim}. As a final remark we show how to 
choose a spanning-tree basis for $\calF(G)$, making the 
proof in \cite{SuWagner} fully constructive for graphic matroids.

To choose a spanning-tree basis, the first step is the same as before: 
list the one codimension faces and parallel classes of edges of $V(G)$. 
Then construct a 0-1 matrix, the rows of which are indexed by the 
(unoriented) circuits of $G$ and the columns by the 2-cut blocks 
$S_{[\epsilon]}$ as follows. Place a $1$ in the field $(C, S_{[\epsilon]})$ if 
and only if $S_{[\epsilon]}\subset C$, which is the case if and only if 
$[\epsilon] \notin F_C$. 
 
A spanning-tree basis of $G$ can be found by the following greedy algorithm. 
To find a spanning tree, one needs to delete edges from a graph until no 
circuits remain, but so that the graph stays connected. Consider the first 
row of the matrix, indexed by the circuit $C_1$. Find the first $1$ in the 
row $C_1$, without loss of generality assume that $S_1 \subset C_1$. 
Imagine deleting (we say imagine, as $G$ is not known to us) one edge from 
$S_1$, thereby breaking the circuit $C_1$. To mark this change, turn all the 
$1$'s in the column of $S_1$ red: all the circuits which contained $S_1$ are 
now broken.
 
Now find the first unbroken circuit (i.e. one which has no red $1$ in its 
row). Delete (in the ``imaginary'' graph $G$) one edge of the first 
2-cut block which participates in it, and turn all the $1$'s in that column 
red. Continue in this manner until no circuits are left intact, i.e. until 
every row has a red $1$ in it.
 
The process will clearly terminate. At the end we know that there are no 
circuits left in $G$, although we don't know $G$. Also, $G$ remains 
connected as only edges which participate in a circuit are ever deleted. 
Hence, what remains is a spanning tree of $G$. The elements of the 
corresponding spanning-tree basis of $\calF(G)$ are those circuits which 
have {\em exactly one edge not in the spanning tree}, i.e., the circuits 
which have {\em only one} red $1$ in their row at the end of the algorithm.

As an illustration, applying the above algorithm to the 
lattice of integer flows in Example~\ref{ex.3D11}, we obtain the matrix
\begin{equation*}
\begin{bmatrix}
1 & 0 & 1 & 0 \\
0 & 0 & 1 & 1 \\
1 & 1 & 0 & 0 \\
0 & 1 & 0 & 1 \\
1 & 0 & 0 & 1 \\
0 & 1 & 1 & 0 \\
\end{bmatrix}
\end{equation*}
This is the same matrix as \eqref{eq.MGmatrix}, except with only one column 
for each two-cut block $S_{\epsilon}$, and we don't need to know what the 
sizes of the blocks are. Carrying out the algorithm leads to the following 
matrix (marking red 1's by $*$ for black and white print):
\begin{equation*}
\begin{bmatrix}
\red{1*} & 0 & \red{1*} & 0 \\
0 & 0 & \red{1*} & 1 \\
\red{1*} & \red{1*} & 0 & 0 \\
0 & \red{1*} & 0 & 1 \\
\red{1*} & 0 & 0 & 1 \\
0 & \red{1*} & \red{1*} & 0 \\
\end{bmatrix}
\end{equation*}
From this we read off the spanning tree basis $\{C_2, C_4, C_5\}$. This basis 
corresponds to the spanning tree $\{e_2, e_4, e_5, e_6\}$ with the notation of 
Figure~\ref{fig.3triangles}, though the algorithm does not output this 
information.
\end{remark}

\subsection*{Acknowledgment}
The authors wish to thank Dror Bar-Natan, Josh Greene, Tony Licata, 
and Brendan McKay for useful conversations. S.G. was supported in part by
in part by the National Science Foundation Grant DMS-14-06419.

\bibliographystyle{hamsalpha}
\bibliography{biblio}
\end{document}